\documentclass[reqno,11pt,centertags]{amsart}
\usepackage{geometry}
\usepackage{amsmath,amsthm,amscd,amssymb,latexsym,upref,stmaryrd}
\usepackage{graphicx}
\usepackage{hyperref}
\usepackage{enumitem}
\usepackage{xcolor}
\newcommand*{\mailto}[1]{\href{mailto:#1}{\nolinkurl{#1}}}
\usepackage[numbers,sort&compress]{natbib}

\newcommand{\beq}{\begin{equation}}
	\newcommand{\eeq}{\end{equation}}
\newcommand{\ba}{\begin{align}}
	\newcommand{\ea}{\end{align}}

 \allowdisplaybreaks
\numberwithin{equation}{section}
\newtheorem{theorem}{Theorem}[section]
\newtheorem{lemma}[theorem]{Lemma}
\newtheorem{corollary}[theorem]{Corollary}

\theoremstyle{definition}

\begin{document}
	
	%\Volume{}
	%\Year{}
	%\DOIsuffix{mana.200310}
	%\pagespan{1}{}
	%\Receiveddate{}
	%\Reviseddate{}
	%\Accepteddate{}
	%\Dateposted{}
	
	\title[Ambarzumyan-type theorem]
	{Ambarzumyan-type theorem for the Sturm-Liouville operator on the lasso graph}
	
	\author[F.~Wang]{Feng Wang}
	\address{School of Mathematics and Statistics, Nanjing University of
		Science and Technology, Nanjing, 210094, Jiangsu, China}
	\email{\mailto{wangfengmath@njust.edu.cn}}
	
	\author[C.~F.~Yang]{CHUAN-FU Yang}
	\address{Department of Mathematics, School of Mathematics and Statistics, Nanjing University of
		Science and Technology, Nanjing, 210094, Jiangsu, People's
		Republic of China}
	\email{\mailto{chuanfuyang@njust.edu.cn}}

	\date{\today}
	
	\begin{abstract}
{We consider the Sturm-Liouville operator on the lasso graph with a segment and a loop joined at one point, which has arbitrary length. The Ambarzumyan's theorem for the operator is proved, which says that if the eigenvalues of the operator coincide with those of the zero potential, then the potential is zero.}
	\end{abstract}
	
	\maketitle
	
\section{introduction}
	Inverse spectral problems consist in recovering the coefficients of an operator from their spectral characteristics. The first inverse spectral result on Sturm-Liouville operators is given by Ambarzumyan \cite{Am}, which describes the following theorem:

\emph{If} $q$ \emph{is a continuous real-valued function on $[0,1]$, and} $\{n^{2}\pi^{2}:n=0,1,2,\ldots\}$ \emph{is the set of eigenvalues of the boundary value problem}
\begin{align}
-y''+q(x)y=\lambda y,\; x\in(0,1),\quad y'(0)=y'(1)=0, \nonumber
\end{align}
\emph{then} $q(x)\equiv0$ on $[0,1]$.

This theorem is called Ambarzumyan$^{\textbf{,}}$s theorem, and has been generalized in many directions (see [4-6, 9-16, 18-19] and other papers). Here we mention Ambarzumyan-type theorems on star graphs \cite{Piv, YHY1}, Ambarzumyan-type theorems on trees \cite{CP, LY}, and Ambarzumyan-type theorems on graphs with cycles \cite{Kiss, YX}.

Differential operators on graphs often appear in natural sciences and engineering (see \cite{BCFK, BK} and the references therein). Such operators can be used to model the motion of quantum particles confined to certain low dimensional structures. In recent years, Ambarzumyan-type theorems on graphs have attracted the attention of many researchers. Some Ambarzumyan-type theorems for the Sturm-Liouville operator on graphs have been achieved in the literatures mentioned above and other works. Most of the works in this direction are devoted to the graphs with equal length edges. It is more difficult to study the Ambarzumyan-type theorems on graphs with unequal length edges. Nevertheless, it is worth noting that C. K. Law and E. Yanagida \cite{LY} studied the Ambarzumyan-type theorems on trees with unequal length edges.

In addition, we have also noticed that an Ambarzumian-type theorem for the Sturm-Liouville operator with a bounded real potential on arbitrary compact graphs is found by Davies \cite{Davies}. In this paper, using the method different from \cite{Davies}, we obtain an Ambarzumian-type theorem
for the Sturm-Liouville operator with a square integrable real potential on the lasso graph with arbitrary edge lengths. Our approach is based on the Hadamard's factorization theorem and the variational principle, which is simpler and can be applied to arbitrary compact graphs.

In this paper we consider the lasso graph $G$ (see Figure \ref{lasso}). The edge $e_{1}$ is a boundary edge of length $l_{1}$,  the edge $e_{2}$ is a loop of length $l_{2}$. Each edge $e_{j}$ is parameterized by the parameter $x_{j}\in[0,l_{j}]$. The value $x_{1}=0$ corresponds to the boundary vertex, and $x_{1}=l_{1}$ corresponds to the internal vertex. For the loop  $e_{2}$, both ends $x_{2}=0$ and $x_{2}=l_{2}$ correspond to the internal vertex.
\begin{figure}
  \centering
  \includegraphics[width=10cm]{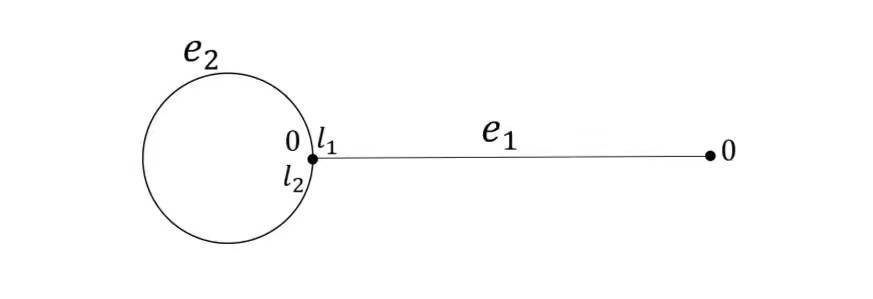}\\
  \caption{Lasso graph G}\label{lasso}
\end{figure}

A function $y$ on $G$ may be represented as a vector function $y=\{y_{j}\}_{j=1,2}$, where the function $y_{j}(x_{j})$, $x_{j}\in[0,l_{j}]$, is defined on the edge $e_{j}$. Consider the following Sturm-Liouville equations on $G$:
\begin{equation}\label{1}
  \varphi(y_{j}):=-y''_{j}(x_{j})+q_{j}(x_{j})y_{j}(x_{j})=\lambda y_{j}(x_{j}),\quad x_{j}\in[0,l_{j}],\quad j=1,2,
\end{equation}
where $q_{j}$, $j=1,2$, are real-valued functions from $L^{2}[0,l_{j}]$, $\lambda$ is the spectral parameter, the functions $y_{j}$, $y'_{j}$, $j=1,2$, are absolutely continuous on $[0,l_{j}]$ and satisfy the following matching conditions in the internal vertex:
\begin{equation}\label{2}
  y_{1}(l_{1})=y_{2}(l_{2})=y_{2}(0),\qquad y'_{1}(l_{1})+y'_{2}(l_{2})-y'_{2}(0)=0.
\end{equation}
Matching conditions (\ref{2}) are called the standard conditions. In electrical circuits, (\ref{2}) expresses Kirchhoff's law; in elastic string network, it expresses the balance of tension, and so on.

Denote $q=\{q_{j}\}_{j=1,2}$ called the potential on $G$. Let us consider the boundary value problem $L(q)$  on $G$ for equation (\ref{1}) with the standard matching conditions (\ref{2}) and with Neumann boundary condition in the boundary vertex:
\begin{equation}\label{3}
  y'_{1}(0)=0.
\end{equation}

It is obvious that the spectrum $\sigma(L(q)$ of the boundary value problem $L(q)$ is a discrete real sequence, bounded from below, diverging to $+\infty$. Let $\sigma(L(q))=\{\lambda_{n}(q)\}_{n=0}^{\infty}$, which can be arranged in an ascending order as (counting with their multiplicities)
\begin{align}
\lambda_{0}(q) \leq \lambda_{1}(q) \leq \cdots  \leq\lambda_{n}(q) \leq \cdots \longrightarrow +\infty. \nonumber
\end{align}

The main result in this paper is as follows.
\begin{theorem}\label{th}
If $\lambda_{n}(q)=\lambda_{n}(0)$ for all $n=0,1,2\cdots$, then $q_{j}(x_{j})=0$ a.e. on $[0,l_{j}]$, $j=1,2$.
\end{theorem}

\section{Analysis of the characteristic function}
	
	In this section we analyze the characteristic function of the  boundary value problem $L(q)$, which plays a key role in the proof of Theorem \ref{th}.

Let $C_{j}(x_{j},\lambda)$, $S_{j}(x_{j},\lambda)$, $j=1,2$, be the solutions of equation (\ref{1}) on the edge $e_{j}$ with the initial conditions
\begin{equation}
  C_{j}(0,\lambda)=S'_{j}(0,\lambda)=1,\qquad C'_{j}(0,\lambda)=S_{j}(0,\lambda)=0.  \nonumber
\end{equation}
For each fixed $x_{j}\in[0,l_{j}]$, the functions $C_{j}^{(\nu)}(x_{j},\lambda)$, $S_{j}^{(\nu)}(x_{j},\lambda)$, $j=1,2$, $\nu=0,1$, are entire in $\lambda$ of order $\frac{1}{2}$.

Then the solutions of equation (\ref{1}) which satisfy condition (\ref{3}) are represented as
\begin{align}\label{4}
       \begin{cases}
          y_{1}(x_{1},\lambda)=A_{1}(\lambda)C_{1}(x_{1},\lambda),  \\
           y_{2}(x_{2},\lambda)=A_{2}(\lambda)C_{2}(x_{2},\lambda)+B_{2}(\lambda)S_{2}(x_{2},\lambda),
        \end{cases}
\end{align}
where $A_{1}(\lambda)$, $A_{2}(\lambda)$ and $B_{2}(\lambda)$ are only dependent on $\lambda$. Substituting (\ref{4}) into matching conditions (\ref{2})  we obtain a linear algebraic system $s$ with respect to $A_{1}(\lambda)$, $A_{2}(\lambda)$ and $B_{2}(\lambda)$. The determinant $\Delta(\lambda)$ of $s$ has the form
\begin{equation}\label{5}
  \Delta(\lambda)=C_{1}(l_{1},\lambda)\Big(C_{2}(l_{2},\lambda)+S'_{2}(l_{2},\lambda)-2\Big)+C'_{1}(l_{1},\lambda)S_{2}(l_{2},\lambda).
\end{equation}
The function $\Delta(\lambda)$ is entire in $\lambda$ of order $\frac{1}{2}$ , and its zeros coincide with the eigenvalues of the boundary value problem $L(q)$. The function $\Delta(\lambda)$ is called the characteristic function for the boundary value problems $L(q)$.

Let $\lambda=\rho^{2}$, $\tau=Im\rho$, then  it follows from \cite{Frei-Yurko-book} that the following asymptotic formulas hold uniformly in $x_{j}\in[0,l_{j}]$:
\begin{align}\label{6}
       \begin{cases}
           S_{j}(x_{j},\lambda)=\frac{\sin\rho x_{j}}{\rho}-\frac{\cos\rho x_{j}}{2\rho^{2}}\int_{0}^{x_{j}}q_{j}(t)dt+o\left(\frac{e^{|\tau|x_{j}}}{\rho^{2}}\right), \\
           S'_{j}(x_{j},\lambda)=\cos\rho x_{j}+\frac{\sin\rho x_{j}}{2\rho}\int_{0}^{x_{j}}q_{j}(t)dt+o\left(\frac{e^{|\tau|x_{j}}}{\rho}\right), \\
           C_{j}(x_{j},\lambda)=\cos\rho x_{j}+\frac{\sin\rho x_{j}}{2\rho}\int_{0}^{x_{j}}q_{j}(t)dt+o\left(\frac{e^{|\tau|x_{j}}}{\rho}\right),\\
           C'_{j}(x_{j},\lambda)=-\rho\sin\rho x_{j}+\frac{\cos\rho x_{j}}{2}\int_{0}^{x_{j}}q_{j}(t)dt+o\left(e^{|\tau|x_{j}}\right),
       \end{cases}
\end{align}
as $|\rho|\rightarrow\infty$. According to (\ref{5}) and (\ref{6}) we have
\begin{align}\label{7}
 \Delta(\lambda)= \Delta_{0}(\lambda)+\frac{2(\sin\rho l_{1})(\cos\rho l_{2}-1)+(\cos\rho l_{1})(\sin\rho l_{2})}{2\rho}[q_{1}] \nonumber \\
 +\frac{2(\cos\rho l_{1})(\sin\rho l_{2})+(\sin\rho l_{1})(\cos\rho l_{2})}{2\rho}[q_{2}]\nonumber \qquad\\
 +o\left(\frac{e^{|\tau|(l_{1}+l_{2})}}{\rho}\right),\qquad |\rho|\rightarrow\infty,\qquad\qquad\qquad\;\;
\end{align}
where
\begin{align}\label{8}
  \Delta_{0}(\lambda)=2(\cos\rho l_{1})(\cos\rho l_{2}-1)-(\sin\rho l_{1})(\sin\rho l_{2}),
\end{align}
\begin{align}\label{9}
  [q_{j}]=\int_{0}^{l_{j}}q_{j}(t)dt,\qquad j=1,2. \qquad\qquad\qquad\qquad
\end{align}
The function $\Delta_{0}(\lambda)$ is entire in $\lambda$ of order $\frac{1}{2}$, and the characteristic function for the boundary value problems $L(0)$ with zero potential.

Let us show that the specification of the spectrum $\sigma(L(q))=\{\lambda_{n}(q)\}_{n=0}^{\infty}$ uniquely determines the characteristic function
$\Delta(\lambda)$. To this end, we consider together with $L(q)$ the boundary value problems $L(\widetilde{q})$ of the same form but with different $\widetilde{q}=\{\widetilde{q}_{j}\}_{j=1,2}$. We agree that if a certain symbol $\beta$ denotes an object related to $L(q)$, then $\widetilde{\beta}$ will denote the analogous object related to $L(\widetilde{q})$.

\begin{lemma}\label{lem1}
If $\lambda_{n}(q)=\lambda_{n}(\widetilde{q})$ for all $n=0,1,2\cdots$, then $\Delta(\lambda)=\widetilde{\Delta}(\lambda)$.
\end{lemma}

\begin{proof}
Denote
\begin{align}
  \lambda_{n}^{1}(0)=\begin{cases}
            \lambda_{n}(0)\qquad if\;\; \lambda_{n}(0)\neq0,\\
           1\qquad\quad\;\;\; if\;\;\lambda_{n}(0)=0,
       \end{cases} \nonumber
\end{align}
where $\lambda_{n}(0)$, $n=0,1,2\cdots$, are eigenvalues of the boundary value problems $L(0)$. By Hadamard's factorization theorem, we have
\begin{equation}\label{10}
  \Delta_{0}(\lambda)=C_{0}\prod_{n=0}^{\infty}\frac{\lambda_{n}(0)-\lambda}{\lambda_{n}^{1}(0)},
\end{equation}
where
\begin{equation}
  C_{0}=\frac{(-1)^{m}}{m!}\left(\frac{\partial^{m}}{\partial\lambda^{m}}\Delta_{0}(\lambda)\right)|_{\lambda=0},\nonumber
\end{equation}
and $m\geq0$ is the multiplicity of the zero eigenvalue of $L(0)$. Note that the infinite product in (\ref{10}) is absolutely convergent (e.g. see Theorem B.2 in \cite{Ges-Sim}).

Denote
\begin{align}\label{11}
  \lambda_{n}^{1}(q)=\begin{cases}
            \lambda_{n}(q)\qquad if\;\; \lambda_{n}(q)\neq0,\\
           1\qquad\quad\;\;\; if\;\;\lambda_{n}(q)=0.
       \end{cases}
\end{align}
Using Hadamard's factorization theorem again, we get
\begin{equation}\label{12}
  \Delta(\lambda)=C\prod_{n=0}^{\infty}\frac{\lambda_{n}(q)-\lambda}{\lambda_{n}^{1}(q)},
\end{equation}
where $C\neq0$ is a constant. Combining the equations (\ref{10}) and (\ref{12}) yields
\begin{equation}\label{13}
  \frac{\Delta(\lambda)}{\Delta_{0}(\lambda)}=\frac{C}{C_{0}}\prod_{n=0}^{\infty}\frac{\lambda_{n}^{1}(0)}{\lambda_{n}^{1}(q)}
  \prod_{n=0}^{\infty}\left(1+\frac{\lambda_{n}(q)-\lambda_{n}(0)}{\lambda_{n}(0)-\lambda}\right).
\end{equation}
Using properties of the characteristic functions and the eigenvalues one gets for large negative $\lambda$ (see \cite{Yur}, Sec. 2.3 and 2.4):
\begin{equation}
  \lim_{\lambda\rightarrow-\infty}\frac{\Delta(\lambda)}{\Delta_{0}(\lambda)}=1.\nonumber
\end{equation}
Consequently by (\ref{13}), we obtain
\begin{equation}\label{14}
C=C_{0}\left[\lim_{\lambda\rightarrow-\infty}\prod_{n=0}^{\infty}\frac{\lambda_{n}^{1}(0)}{\lambda_{n}^{1}(q)}
  \prod_{n=0}^{\infty}\left(1+\frac{\lambda_{n}(q)-\lambda_{n}(0)}{\lambda_{n}(0)-\lambda}\right)\right]^{-1}.
\end{equation}

Similarly,
\begin{equation}\label{15}
\widetilde{\Delta}(\lambda)=\widetilde{C}\prod_{n=0}^{\infty}\frac{\lambda_{n}(\widetilde{q})-\lambda}{\lambda_{n}^{1}(\widetilde{q})},
\end{equation}
where
\begin{equation}\label{16}
\widetilde{C}=C_{0}\left[\lim_{\lambda\rightarrow-\infty}\prod_{n=0}^{\infty}\frac{\lambda_{n}^{1}(0)}{\lambda_{n}^{1}(\widetilde{q})}
  \prod_{n=0}^{\infty}\left(1+\frac{\lambda_{n}(\widetilde{q})-\lambda_{n}(0)}{\lambda_{n}(0)-\lambda}\right)\right]^{-1},
\end{equation}
and
\begin{align}\label{16-1}
  \lambda_{n}^{1}(\widetilde{q})=\begin{cases}
            \lambda_{n}(\widetilde{q})\qquad if\;\; \lambda_{n}(\widetilde{q})\neq0,\\
           1\qquad\quad\;\;\; if\;\;\lambda_{n}(\widetilde{q})=0.
       \end{cases}
\end{align}

According to (\ref{11}), (\ref{14}), (\ref{16}) and (\ref{16-1}), together with assumptions of the lemma, we have $\lambda_{n}^{1}(q)=\lambda_{n}^{1}(\widetilde{q})$ and $C=\widetilde{C}$. Thus, it follows from (\ref{12}) and (\ref{15}) that $\Delta(\lambda)=\widetilde{\Delta}(\lambda)$.
\end{proof}

From Lemma \ref{lem1}, we can immediately get the following corollary.
\begin{corollary}\label{cor}
If $\lambda_{n}(q)=\lambda_{n}(0)$ for all $n=0,1,2\cdots$, then $\Delta(\lambda)=\Delta_{0}(\lambda)$.
\end{corollary}

\section{Proof of Theorem 1.1}
	
	This section is devoted to the proof of Theorem \ref{th}. The following lemma plays an important role in the proof of Theorem \ref{th}.

\begin{lemma}\label{lem2}
If $\Delta(\lambda)=\Delta_{0}(\lambda)$, then $[q_{1}]=[q_{2}]=0$.
\end{lemma}

\begin{proof}
(1) When $l_{1}=l_{2}=l$, from (\ref{7}), together with $\Delta(\lambda)=\Delta_{0}(\lambda)$, we have
\begin{align}
\frac{(\sin\rho l)(3\cos\rho l-2)}{2\rho}[q_{1}]+\frac{3(\sin\rho l)(\cos\rho l)}{2\rho}[q_{2}]
 =o\left(\frac{e^{2|\tau|l}}{\rho}\right),\quad |\rho|\rightarrow\infty. \nonumber
\end{align}
In the above estimate, taking $\rho=\frac{(n+\frac{1}{2})\pi}{l}$  to get $[q_{1}]=0$ and $\rho=\frac{2n\pi+\arccos\frac{2}{3}}{l}$  to get $[q_{2}]=0$.

(2) When $l_{1}\neq l_{2}$, it can be proved in two cases.

(i) Let $\frac{l_{2}}{l_{1}}$ be an irrational number. From (\ref{7}), together with $\Delta(\lambda)=\Delta_{0}(\lambda)$, we have

\begin{align}\label{17}
 \frac{2(\sin\rho l_{1})(\cos\rho l_{2}-1)+(\cos\rho l_{1})(\sin\rho l_{2})}{2\rho}[q_{1}] \nonumber \\
 +\frac{2(\cos\rho l_{1})(\sin\rho l_{2})+(\sin\rho l_{1})(\cos\rho l_{2})}{2\rho}[q_{2}] \quad\nonumber \\
 =o\left(\frac{e^{|\tau|(l_{1}+l_{2})}}{\rho}\right),\quad |\rho|\rightarrow\infty.\qquad\qquad\qquad\;
\end{align}
Taking $\rho=\frac{2n\pi}{l_{1}}$ in the estimate (\ref{17}), we obtain
\begin{align}\label{18}
 \sin\left(\frac{l_{2}}{l_{1}}2n\pi\right)[q_{1}]+2\sin\left(\frac{l_{2}}{l_{1}}2n\pi\right)[q_{2}]=o(1),\qquad n\rightarrow\infty.
\end{align}
Note that
\begin{align}\label{19}
 \left(\sin\left(\frac{l_{2}}{l_{1}}2n\pi\right)\right)^{-1}=O(1),\qquad n\rightarrow\infty,
\end{align}
since $\frac{l_{2}}{l_{1}}$ is an irrational number. Combining (\ref{18}) and (\ref{19}) yields
\begin{align}
 [q_{1}]+2[q_{2}]=o(1),\qquad n\rightarrow\infty, \nonumber
\end{align}
which follows
\begin{align}\label{20}
 [q_{1}]+2[q_{2}]=0.
\end{align}
Taking $\rho=\frac{(2n+1)\pi}{l_{2}}$ in the estimate (\ref{17}), we obtain
\begin{align}\label{21}
 2\sin\left(\frac{l_{1}}{l_{2}}(2n+1)\pi\right)[q_{1}]+\sin\left(\frac{l_{1}}{l_{2}}(2n+1)\pi\right)[q_{2}]=o(1),\qquad n\rightarrow\infty.
\end{align}
Note that
\begin{align}\label{22}
 \left(\sin\left(\frac{l_{1}}{l_{2}}(2n+1)\pi\right)\right)^{-1}=O(1),\qquad n\rightarrow\infty,
\end{align}
since $\frac{l_{1}}{l_{2}}$ is an irrational number.  Combining (\ref{21}) and (\ref{22}) yields
\begin{align}
 2[q_{1}]+[q_{2}]=o(1),\qquad n\rightarrow\infty, \nonumber
\end{align}
which follows
\begin{align}\label{23}
 2[q_{1}]+[q_{2}]=0.
\end{align}
It is easy to see that equations (\ref{20}) and (\ref{23}) imply $[q_{1}]=[q_{2}]=0$.

(ii) Let $\frac{l_{2}}{l_{1}}$ be a rational number. Without losing generality, we assume $l_{1}=k_{1}l$, $l_{2}=k_{2}l$, where $k_{1}$, $k_{2}$
are two positive integers and relatively prime. From (\ref{7}), together with $\Delta(\lambda)=\Delta_{0}(\lambda)$, we have

\begin{align}\label{24}
 \frac{2(\sin\rho k_{1}l)(\cos\rho k_{2}l-1)+(\cos\rho k_{1}l)(\sin\rho k_{2}l)}{2\rho}[q_{1}]\!\!\!\!\! \nonumber \\
 +\frac{2(\cos\rho k_{1}l)(\sin\rho k_{2}k)+(\sin\rho k_{1}l)(\cos\rho k_{2}l)}{2\rho}[q_{2}] \nonumber \\
 =o\left(\frac{e^{|\tau|(k_{1}+k_{2})l}}{\rho}\right),\quad |\rho|\rightarrow\infty.\qquad\qquad\qquad\;\;\;\:
\end{align}
When $k_{1}=k_{2}$ then $[q_{1}]=[q_{2}]=0$ is obvious from Case (1). The following is divided into cases $k_{2}<k_{1}$ and $k_{2}>k_{1}$.

a) Case $k_{2}<k_{1}$.
In the estimate (\ref{24}), taking $\rho=\frac{(2k_{1}n+1)\pi}{k_{1}l}$ and let $n\rightarrow\infty$, one can get
\begin{align}\label{25}
 \sin\left(\frac{k_{2}\pi}{k_{1}}\right)[q_{1}]+2\sin\left(\frac{k_{2}\pi}{k_{1}}\right)[q_{2}]=0,
\end{align}
which follows
\begin{equation}\label{26}
  [q_{1}]=-2[q_{2}],
\end{equation}
since $0<\sin\left(\frac{k_{2}\pi}{k_{1}}\right)<1$.
In the estimate (\ref{24}), taking $\rho=\frac{(2k_{1}n+\frac{1}{2})\pi}{k_{1}l}$ and let $n\rightarrow\infty$, one can get
\begin{align}\label{27}
 2\left(\cos\left(\frac{k_{2}\pi}{2k_{1}}\right)-1\right)[q_{1}]+\cos\left(\frac{k_{2}\pi}{2k_{1}}\right)[q_{2}]=0.
\end{align}
Substituting (\ref{26}) into (\ref{27}), we have
\begin{align}
 \left(4-3\cos\left(\frac{k_{2}\pi}{2k_{1}}\right)\right)[q_{2}]=0. \nonumber
\end{align}
Since $4-3\cos\left(\frac{k_{2}\pi}{2k_{1}}\right)\neq0$, this yields $[q_{2}]=0$. Thus $[q_{1}]=0$ from (\ref{26}).

b) Case $k_{2}>k_{1}$.
In the estimate (\ref{24}), taking $\rho=\frac{(2k_{2}n+1)\pi}{k_{2}l}$ and let $n\rightarrow\infty$, one can get
\begin{align}\label{28}
 -4\sin\left(\frac{k_{1}\pi}{k_{2}}\right)[q_{1}]-\sin\left(\frac{k_{1}\pi}{k_{2}}\right)[q_{2}]=0,
\end{align}
which follows
\begin{equation}\label{29}
  [q_{2}]=-4[q_{1}],
\end{equation}
since $0<\sin\left(\frac{k_{1}\pi}{k_{2}}\right)<1$. In the estimate (\ref{24}), taking $\rho=\frac{(2k_{2}n+\frac{1}{2})\pi}{k_{2}l}$ and let $n\rightarrow\infty$, one can get
\begin{align}\label{30}
 \left(-2\sin\left(\frac{k_{1}\pi}{2k_{2}}\right)+\cos\left(\frac{k_{1}\pi}{2k_{2}}\right)\right)[q_{1}]+2\cos\left(\frac{k_{1}\pi}{2k_{2}}\right)[q_{2}]=0.
\end{align}
Substituting (\ref{29}) into (\ref{30}), we have
\begin{align}
 \left(2\sin\left(\frac{k_{1}\pi}{2k_{2}}\right)+7\cos\left(\frac{k_{1}\pi}{2k_{2}}\right)\right)[q_{1}]=0. \nonumber
\end{align}
Since $2\sin\left(\frac{k_{1}\pi}{2k_{2}}\right)+7\cos\left(\frac{k_{1}\pi}{2k_{2}}\right)\neq0$, this yields $[q_{1}]=0$.
Thus $[q_{2}]=0$ from (\ref{29}).
The proof of Lemma \ref{lem2} is complete.
\end{proof}

Let us introduce the Hilbert space $H:=L^{2}(0,l_{1})\oplus L^{2}(0,l_{2})$ with the inner product
\begin{align}
(f,g)=\sum_{j=1}^{2}\int_{0}^{l_{j}}f_{j}(x)\overline{g_{j}(x)}dx,\nonumber
\end{align}
where $f=(f_{1},,f_{2})^{T}\in H$, $g=(g_{1},g_{2})^{T}\in H$, and $g^{T}$ denotes the transpose of the vector $g$. The domain of self-adjoint operator $L(q)$ is
\begin{align}
D(L(q))=\Big\{y=(y_{1},y_{2})^{T}\in H\:|\:y_{j}\in AC(0,l_{j}),\;y'_{j}\in AC[0,l_{j}],\;j=1,2, \quad\quad\nonumber\\
\; satisfying \;(\ref{2})\; and\; (\ref{3}),\;\varphi(y):=(\varphi(y_{1}),\varphi(y_{2}))^{T}\in H\Big\}, \!\!\!\nonumber
\end{align}
where $AC[0,l_{j}]$ represents a set of all absolutely continuous functions on $[0,l_{j}]$.
\\

 \textbf{\textit{Proof of Theorem \ref{th}}}. It is readily verified that the operator $L(0)$ is non-negative and $0\in\sigma(L(0))$, so zero is its smallest eigenvalue, namely $\lambda_{0}(0)=0$.

Next we show that $y_{0}=(1,1)^{T}$ is an eigenfunction of $L(q)$ corresponding to the eigenvalue zero. By the variational principle, we obtain
\begin{align}
0=\lambda_{0}(0)=\lambda_{0}(q)=\inf_{0\neq y\in D(L(q))}\frac{(\varphi(y),y)}{(y,y)} \qquad\qquad\qquad\qquad\quad\nonumber\\
=\inf_{0\neq y\in D(L(q))}\frac{\sum\limits_{j=1}^{2}\Big(-\int_{0}^{l_{j}}y''_{j}(x)\overline{y_{j}(x)}dx
+\int_{0}^{l_{j}}q_{j}(x)|y_{j}(x)|^{2}dx\Big)}{\sum\limits_{j=1}^{2}\int_{0}^{l_{j}}|y_{j}(x)|^{2}dx}.\!\!\!\!\!\nonumber
\end{align}
Now  $y_{0}\in D(L(q))$ is obvious, and so
\begin{align}
0\leq \frac{(\varphi(y_{0}),y_{0})}{(y_{0},y_{0})}=\frac{\sum\limits_{j=1}^{2}\int_{0}^{l_{j}}q_{j}(x)dx}{\sum\limits_{j=1}^{2}l_{j}}
=\frac{[q_{1}]+[q_{2}]}{l_{1}+l_{2}},\nonumber
\end{align}
by Corollary \ref{cor} and Lemma \ref{lem2}, the right hand side is exactly zero, which implies that the test function $y_{0}$ is an eigenfunction of $L(q)$ corresponding to the eigenvalue zero. Substituting $y_{0}$, which is the eigenfunction of eigenvalue zero, into equations (\ref{1}), we obtain $q_{j}=0$ in $L^{2}[0,l_{j}]$, $j=1,2$. The proof is finished. $\square$

\qquad
	
%%%%%%%%%%%%%%%%%%%%%%%%%%%%%%%%%%%%%
\noindent {\bf Acknowledgments.}
	This work was supported in part by  the National Natural Science Foundation of China (11871031) and the National Natural Science Foundation of Jiang Su (BK20201303).

\end{document}